\documentclass[12pt]{article}
\usepackage{latexsym,bm}
\usepackage{mathrsfs,amsmath,amssymb,amsthm}
\usepackage[dvipdfm]{graphicx}
\theoremstyle{plain}
\newtheorem{theorem}{Theorem}

\newtheorem{lemma}[theorem]{Lemma}
\newtheorem{corollary}[theorem]{Corollary}

\theoremstyle{definition}

\theoremstyle{remark}

\usepackage{amscd}

\newcommand{\la}{\langle}
\newcommand{\ra}{\rangle}

\makeatletter
\newcommand{\trace}{\mathop{\operator@font Trace}}
\newcommand{\vspan}{\mathop{\operator@font Span}}
\newcommand{\Int}{\mathop{\operator@font Int}}
\newcommand{\grad}{\mathop{\operator@font grad}}
\makeatother

\begin{document}

\title{Semi-parallel real hypersurfaces in complex two-plane Grassmannians
\thanks{This work was supported in part by the UMRG research grant (Grant No. RG163/11AFR)}
}

\author{Tee-How \textsc{Loo}
\\
Institute of Mathematical Sciences, University of Malaya \\
50603 Kuala Lumpur, Malaysia.	\\ 
\ttfamily{looth@um.edu.my}
}
\date{}
\maketitle

\abstract{We prove that there does not exist any semi-parallel real hypersurface in complex two-plane
Grassmannians. With this result, the nonexistence of recurrent real hypersurfaces in complex two-plane Grassmannians can also be proved.}



\section{Introduction}
The notion of semi-parallel submanifolds, as a generalization of parallel submanifolds 
(submanifolds with parallel second fundamental form), was first studied by Deprez in \cite{deprez2}. 
A submanifold $M$ in a Riemannian manifold is said to be \emph{semi-parallel} if 
the second fundamental form $h$ satisfies $\bar R\cdot h=0$, where $\bar R$ is the curvature tensor corresponding to the van der Waerden-Bortolotti connection.

It was proved in \cite{deprez} that a semi-parallel hypersurface in a Euclidean space is either flat; parallel; or is an open part of a round cone or of a product of a round cone and a linear subspace.
When the ambient space is a sphere or real hyperbolic space, Dillen showed that a semi-parallel 
hypersurface is either an open part of a flat surface, parallel or an open part of a rotation hypersurfaces of certain helices \cite{dillen}.
A thorough survey on the study of semi-parallel submanifolds in a real space form can be found in \cite{lumiste}.

When the ambient space is a non-flat complex space form, 
parallel submanifolds were classified by Naitoh \cite{naitoh}.
As a result, the shape operator of a real hypersurface cannot be parallel.
The existence problem of semi-parallel real hypersurfaces was first studied by Maeda \cite{maeda}
for complex projective spaces of complex dimension greater than two, followed by 
Niegerball and Ryan \cite{ryan} for non-flat complex space forms of complex dimension two; and had completely been solved by Ortega \cite{ortega}. 
\begin{theorem}[\cite{ortega}]\label{thm:ortega}
There does not exist any semi-parallel real hypersurface in a non-flat complex space form.  
\end{theorem}

For codimension greater than one, Kon \cite{kon} proved that there does not exist any semi-parallel proper CR-submanifold in a complex projective space with semi-flat normal connection and with CR-dimension greater than one.
As a byproduct of their main results in \cite{lobos}, Chac\'{o}n and Lobos have classified all semi-parallel Lagrangian surfaces in a complex space form.

The study of Riemannian submanifolds has been extended to ambient spaces which are symmetric spaces other than real space forms and complex space forms. In particular, the study of real hypersurfaces in complex two-plane Grassmannians $G_2(\mathbb C_{m+2})$ has been an active field recently.

$G_2(\mathbb C_{m+2})$ is the unique compact irreducible Riemannian symmetric space with both a Kaehler structure $J$ and a quaternionic Kaehler structure $\mathfrak J$.
These two geometric structures induce on its real hypersurfaces $M$ a (local) almost contact $3$-structure
$(\phi_a,\xi_a,\eta_a)$, $a\in\{1,2,3\}$ and almost contact structure $(\phi, \xi,\eta)$.
In \cite{berndt-suh}, Berndt and Suh classified all real hypersurfaces $M$ in $G_2(\mathbb C_{m+2})$ on which both $\vspan\{\xi\}$ and $\mathfrak D^\perp$ are invariant under the shape operator $A$ of $M$,
where $\mathfrak D^\perp$ is the distribution on $M$ defined by
$\mathfrak D^\perp_x=\vspan\{\xi_1,\xi_2,\xi_3\}$, $x\in M$. 
Such real hypersurfaces can be expressed as tubes around totally geodesic submanifolds $G_2(\mathbb C_{m+1})$
or $\mathbb  HP_{m/2}$ in $G_2(\mathbb C_{m+2})$
(cf. Theorem~\ref{thm:classification}). 

Since then, a number of interesting results along this line have been obtained.
For instance, 
the characterizations of real hypersurfaces under certain nice relationships between the shape operator $A$ and 
the almost contact structure $\phi$ (see \cite{berndt-suh2}, \cite{suh5}); and some recent papers
(see \cite{machado2}, \cite{machado1}, \cite{suh6}).
In \cite{suh1}, Suh proved the following result. 
\begin{theorem}[\cite{suh1}]\label{thm:suh}
There does not exist any parallel real hypersurface in $G_2(\mathbb C_{m+2})$, $m\geq3$. 
\end{theorem}

Motivated by Theorem~\ref{thm:ortega} and Theorem~\ref{thm:suh}, it is natural to ask if there are any semi-parallel real hypersurfaces in $G_2(\mathbb C_{m+2})$. 
In this paper, we shall answer this question in negative, that is, 
\begin{theorem}	\label{thm:main}
There does not exist any semi-parallel real hypersurface in $G_2(\mathbb C_{m+2})$, $m\geq 3$.
\end{theorem}

Let $\mathcal E$ be a vector bundle over a manifold $M$.
A nonzero $\mathcal E$-valued tensor field $F$ of type $(r,s)$ on $M$ is said to be \emph{recurrent} if there exists a $1$-form $\omega$ in $M$ such that $\bar\nabla F=F\otimes\omega$,
where $\bar\nabla$ is the var der Waerden-Bortolotti connection. In particular, if $\omega=0$ then $F$ is parallel. 
Some geometric interpretations of a manifold $M$ with recurrent curvature tensor in terms of holonomy group were given in \cite{kobayashi}, \cite{wong}.

A submanifold of a Riemannian manifold is said to be \emph{recurrent} if its second fundamental form is recurrent.
The problem of determining the existence of (or classifying) recurrent real hypersurfaces in $G_2(\mathbb C_{m+2})$ has been considered and solved partially. In \cite{suh2}, the nonexistence of recurrent real hypersurfaces was proved under an additional assumption of $\mathfrak D$-invariance of the shape operator.
Kim, Lee and Yang proved in \cite{kim-lee-yang} that there does not exist any Hopf hypersurface with recurrent shape operator.
Recall that a real hypersurface is said to be \emph{Hopf} if the Reeb vector field $\xi$ is principal.

The second objective of this paper is to study the existence of recurrent real hypersurfaces in $G_2(\mathbb C_{m+2})$.
We first show that a recurrent symmetric tensor field $F$ of type $(1,1)$ on a Riemannian manifold  
is necessarily semi-parallel (cf. Theorem~\ref{thm:recurrent}). With this result and Theorem~\ref{thm:main}, we prove the nonexistence of recurrent real hypersurfaces in $G_2(\mathbb C_{m+2})$, $m\geq3$
(cf. Corollary~\ref{cor:recurrent}). 
This improves the results of Suh \cite{suh2} and Kim, Lee and Yang \cite{kim-lee-yang} mentioned above. 

This paper is organized as follows.
In Section 2, we recall some basic properties for $G_2(\mathbb C_{m+2})$ and its real hypersurfaces $M$. 
In Section 3, we first introduce a local symmetric tensor field $\theta_a$ of type $(1,1)$ on $M$, and then
derive some of its properties.
The proof of Theorem~\ref{thm:main} will be given in the next section.
In the last section, we prove the nonexistence of recurrent real hypersurfaces in $G_2(\mathbb C_{m+2})$.

\section{Real hypersurfaces in $G_2(\mathbb C_{m+2})$}
In this section we state some structural equations as well as some known results in the theory of real hypersurfaces in complex two-plane Grassmannians. 
We begin with some basic properties of complex two-plane Grassmannians (cf. \cite{berndt}), which are needed in our paper.

By $G_2(\mathbb C_{m+2})$,
we denote the set of all complex two-dimensional linear subspaces in $\mathbb C_{m+2}$.
Note that $G_2(\mathbb C_3)$ is isometric to the complex projective space $\mathbb CP_2(8)$
and $G_2(\mathbb C_4)$ is isometric to the real Grassmannian $G^+_2(\mathbb R^6)$ of oriented two-dimensional linear subspaces in $\mathbb R^8$. In this paper, we only consider $m\geq3$. 

Denote by $\la,\ra$ the Riemannian metric, $J$ the Kaehler structure and $\mathfrak J$ the quarternionic Kaehler structure on $G_2(\mathbb C_{m+2})$.
For each $x\in G_2(\mathbb C_{m+2})$, we denote by $\{J_1,J_2,J_3\}$ a canonical local basis of $\mathfrak J$
on a neighborhood $\mathcal U$ of $x$ in $G_2(\mathbb C_{m+2})$, that is, each $J_a$ is a local almost Hermitian structure 
such that
\begin{align}\label{eqn:quarternion}
J_aJ_{a+1}=J_{a+2}=-J_{a+1}J_a, \quad a\in\{1,2,3\}.  
\end{align}
Here, the index is taken modulo three.
Denote by $\hat\nabla$ the Levi-Civita connection of $G_2(\mathbb C_{m+2})$.
There exist local $1$-forms $q_1$, $q_2$ and $q_3$ such that
\[
\hat\nabla_XJ_a=q_{a+2}(X)J_{a+1}-q_{a+1}(X)J_{a+2}
\]
for any $X\in T_x	G_2(\mathbb C_{m+2})$, that is, $\mathfrak J$ is parallel with respect to $\hat\nabla$.
The Kaehler structure $J$ and quarternionic Kaehler structure $\mathfrak J$ are related by 
\begin{align}\label{eqn:JJa}
JJ_a=J_aJ; \quad \trace{(JJ_a)}=0, \quad a\in\{1,2,3\}.
\end{align}

The Riemannian curvature tensor $\hat R$ of $G_2(\mathbb C_{m+2})$ is locally given by
\begin{align}\label{eqn:hatR}
\hat R(X,Y)Z=&\la Y,Z\ra X-\la X,Z\ra Y+\la JY,Z\ra JX-\la JX,Z\ra JY-2\la JX,Y\ra JZ	\nonumber\\
&+\sum_{a=1}^3\{\la J_aY,Z\ra J_aX-\la J_aX,Z\ra J_aY-2\la J_aX,Y\ra J_aZ							\nonumber\\
&+\la JJ_aY,Z\ra JJ_aX-\la JJ_aX,Z\ra JJ_aY\}.
\end{align}
for all $X$, $Y$ and $Z\in T_xG_2(\mathbb C_{m+2})$.

For a nonzero vector $X\in T_xG_2(\mathbb C_{m+2})$, we denote by $\mathbb CX=\vspan\{X,JX\}$, $\mathfrak JX=\{J'X|J'\in\mathfrak J_x\}$, 
$\mathbb HX=\vspan\{X\}\oplus\mathfrak JX$, and 
$\mathbb H\mathbb CX$ the subspace spanned by $\mathbb HX$ and $\mathbb HJX$.  
If $JX\in\mathfrak JX$, we denote by $\mathbb C^\perp X$ the orthogonal complement of $\mathbb CX$ in $\mathbb HX$.

Let $M$ be a connected oriented real hypersurface isometrically immersed in $G_2(\mathbb C_{m+2})$, $m\geq3$, $N$ a unit normal vector field on $M$. The Riemannian metric on $M$ is denoted by the same $\la,\ra$.
A canonical local basis $\{J_1,J_2,J_3\}$ of $\mathfrak J$ on  $G_2(\mathbb C_{m+2})$ induces  a local almost contact metric $3$-structure $(\phi_a,\xi_a,\eta_a,\la,\ra)$ on $M$ by  
\begin{align*}	
J_aX=\phi_a X+\eta_a(X)N,	\quad J_aN=-\xi, \quad \eta_a(X)=\la\xi_a,X\ra
\end{align*}
for any $X\in TM$. 
It follows from (\ref{eqn:quarternion}) that  
\begin{align*}
&\phi_a\phi_{a+1}-\xi_a\otimes\eta_{a+1}=\phi_{a+2}=-\phi_{a+1}\phi_a+\xi_{a+1}\otimes\eta_a \\ 
&\phi_a\xi_{a+1}=\xi_{a+2}=-\phi_{a+1}\xi_a.  
\end{align*}
Denote by $(\phi, \xi,\eta,\la,\ra)$ the almost contact metric structure on $M$ induced by $J$, that is, 
\begin{align*}	
JX=\phi X+\eta(X)N,	\quad JN=-\xi, \quad \eta(X)=\la\xi,X\ra. 
\end{align*}
The vector field $\xi$ is known as the \emph{Reeb vector field}. 
A real hypersurface $M$ is said to be \emph{Hopf} if $\xi$ is principal.

It follows from (\ref{eqn:JJa}) that the two structures $(\phi,\xi,\eta,\la,\ra)$ and $(\phi_a,\xi_a,\eta_a,\la,\ra)$
are related as follows  
\begin{align*}
\phi_a\phi-\xi_a\otimes\eta=\phi\phi_a-\xi\otimes\eta_a; \quad \phi\xi_a=\phi_a\xi.
\end{align*}

Denote by $\nabla$ the Levi-Civita connection and $A$ the shape operator on $M$. Then 
\begin{align*}
&(\nabla_{X} \phi)Y=\eta(Y)AX-\la AX,Y\ra\xi, \quad \nabla_X \xi = \phi AX \label{eqn:delxi}\\
&(\nabla_{X}\phi_a)Y=\eta_a(Y)AX-\la AX,Y\ra\xi_{a}+q_{a+2}(X)\phi_{a+1}Y-q_{a+1}(X)\phi_{a+2}Y \\
&\nabla_X \xi_a = \phi_a AX+q_{a+2}(X)\xi_{a+1}-q_{a+1}(X)\xi_{a+2} 
\end{align*}
for any $X,Y\in TM$. From these formulas, we have
\begin{align*}
X\eta(\xi_a)
&=\la\nabla_X\xi,\xi_a\ra+\la\xi,\nabla_X\xi_a\ra		\nonumber\\
&=-2\la A\phi\xi_a,X\ra+\eta(\xi_{a+1})q_{a+2}(X)-\eta(\xi_{a+2})q_{a+1}(X). 
\end{align*}
We define a distribution $\mathfrak D^\perp$ on $M$ by 
$\mathfrak D^\perp_x:=\vspan\{\xi_1,\xi_2,\xi_3\}$, $x\in M$, and denote by $\mathfrak D$ its orthogonal complement in
$TM$.
If $\xi\in\mathfrak D$ at each point in $M$ then $\eta(\xi_a)=0$, for $a\in\{1,2,3\}$, and so by the above equation, we obtain
\begin{lemma}\label{lem:Xua}
Let $M$ be a real hypersurface in $G_2(\mathbb C_{m+2})$. If $\xi$ is tangent to $\mathfrak D$ then $A\phi\xi_a=0$, for $a\in\{1,2,3\}$.
\end{lemma} 

Finally we state some well-known results.
\begin{theorem}[\cite{berndt-suh}]\label{thm:classification}
Let $M$ be a connected real hypersurface in $G_2(\mathbb C_{m+2})$, $m\geq3$. 
Then both $\vspan\{\xi\}$ and $\mathfrak D^\perp$ are invariant under the shape operator of $M$ if and only if
\begin{enumerate}
\item[(A)] $M$ is an open part of a tube around a totally geodesic $G_2(\mathbb C_{m+1})$ in
$G_2(\mathbb C_{m+2})$, or
\item[(B)] $m$ is even, say $m=2n$, and $M$ is an open part of a tube around a totally
geodesic $\mathbb  HP_n$ in $G_2(\mathbb C_{m+2})$.
\end{enumerate}
\end{theorem}

\begin{theorem}[\cite{berndt-suh}]\label{thm:A}
Let $M$ be a real hypersurface of type $A$ in $G_2(\mathbb C_{m+2})$. 
Then
$\xi\in\mathfrak D^\perp$ at each point of $M$.
Suppose $J_1\in\mathfrak J$ such that $J_1N=JN$. Then 
$M$ has three (if $r=\pi/2\sqrt8$) or four (otherwise) distinct constant principal curvatures
\[
\alpha=\sqrt8\cot(\sqrt8r),\ 
\beta=\sqrt2\cot(\sqrt2r),\
\lambda=-\sqrt2\tan(\sqrt2r),\
\mu=0
\]
with some $r\in]0,\pi/\sqrt8[$\:. The corresponding multiplicities are
\[
m(\alpha)=1,\
m(\beta)=2,\
m(\lambda)=2m-2=m(\mu)
\]
and the corresponding eigenspaces are
\begin{align*}
T_\alpha=&\vspan\{\xi\},\  T_\beta=\mathbb C^\perp \xi,\\
T_\lambda=&\{X:X\perp\mathbb H\xi,\ JX=J_1X\},\\
T_\mu=&\{X:X\perp\mathbb H\xi,\ JX=-J_1X\}.
\end{align*}
\end{theorem}

\begin{theorem}[\cite{berndt-suh}]\label{thm:B}
Let $M$ be a real hypersurface of type $B$ in $G_2(\mathbb C_{m+2})$. 
Then 
$\xi\in\mathfrak D$ at each point of $M$,
$m$ is even, say $m=2n$, and 
$M$ has five distinct constant principal curvatures
\[
\alpha=-2\tan(2r),\ 
\beta=2\cot(2r),\
\gamma=0,\
\lambda=\cot(r),\
\mu=-\tan(r)
\]
with some $r\in]0,\pi/4[$. The corresponding multiplicities are
\[
m(\alpha)=1,\
m(\beta)=3=m(\gamma),\
m(\lambda)=4n-4=m(\mu)
\]
and the corresponding eigenspaces are
\[
T_\alpha=\vspan\{\xi\},\
T_\beta=\mathfrak JJ\xi,\
T_\gamma=\mathfrak J\xi,\
T_\lambda,\
T_\mu,
\]
where $T_\lambda\oplus T_\mu=(\mathbb H\mathbb C\xi)^\perp$, $\mathfrak JT_\lambda=T_\lambda$,
$\mathfrak JT_\mu=T_\mu$, $JT_\lambda=T_\mu$.
\end{theorem}

\begin{theorem}[\cite{lee-suh}]\label{thm:lee-suh}
Let $M$ be a connected orientable Hopf real hypersurface in
$G_2(\mathbb C_{m+2})$, $m\geq3$. Then the Reeb vector $\xi$ belongs to the distribution $\mathfrak D$ if and
only if M is locally congruent to an open part of a real hypersurface of type $B$.
\end{theorem}


\section{The symmetric tensor fields $\theta_a$}
In this section, we introduce a local symmetric endomorphism $\theta_a$ in $TM$  for real hypersurfaces $M$ in $G_2(\mathbb C_{m+2})$.
With the notion $\theta_a$, some fundamental identities such as the Gauss equation, to certain extent, could be expressed in a comparatively compact form.
Besides, it possesses some nice characteristics, which are crucial in the proof of our main result.

Let $M$ be a real hypersurface in $G_2(\mathbb C_{m+2})$, $m\geq3$.
Corresponding to each canonical local basis $\{J_1,J_2,J_3\}$ of $\mathfrak J$, 
we define a local endomorphism $\theta_a$ on $TM$ by
\begin{align}\label{eqn:theta}
\theta_aX:=\tan(J_aJX)=\phi_a\phi X-\eta(X)\xi_a=\phi\phi_a X-\eta_a(X)\xi.
\end{align}

Let $R$ be the curvature tensor of $M$. It follows from (\ref{eqn:hatR}) that the equation of Gauss is given by
\begin{align*}
&R(X,Y)Z=\la Y,Z\ra X-\la X,Z\ra Y+\la\phi Y,Z\ra\phi X-\la\phi X,Z\ra\phi Y
																																	-2\la\phi X,Y\ra\phi Z	\nonumber\\
&+\sum_{a=1}^3\{\la\phi_aY,Z\ra\phi_aX-\la\phi_aX,Z\ra \phi_aY-2\la\phi_aX,Y\ra\phi_aZ							\nonumber\\
&+\la\theta_aY,Z\ra\theta_aX-\la\theta_aX,Z\ra\theta_aY\}+\la AY,Z\ra AX-\la AX,Z\ra AY.
\end{align*}
for any $X,Y,Z\in TM$.
Next, we derive some properties of $\theta_a$.
\begin{lemma}\label{lem:theta}
\begin{enumerate}
\item[(a)] $\theta_a$ is symmetric,
\item[(b)] $\trace{(\theta_a)}=\eta(\xi_a)$,
\item[(c)] $\theta_a^2X=X-\la X,\phi \xi_a\ra\phi\xi_a$, for all $X\in TM$,
\item[(d)] $\theta_a\xi=-\xi_a; \quad \theta_a\xi_a=-\xi; \quad \theta_a\phi\xi_a=\eta(\xi_a)\phi\xi_a$, 
\item[(e)] $\theta_a\xi_{a+1}= \phi\xi_{a+2}=-\theta_{a+1}\xi_a$,
\item[(f)] $\theta_a\phi\xi_{a+1}=-\xi_{a+2}+\eta(\xi_{a+1})\phi\xi_a$,
\item[(g)] $\theta_{a+1}\phi\xi_a= \xi_{a+2}+\eta(\xi_a)\phi\xi_{a+1}$.
\end{enumerate}
\end{lemma}
\begin{proof}
For any $X,Y\in TM$,
\[
\la\theta_aX,Y\ra-\la X,\theta_aY\ra
=\la\phi\phi_aX,Y\ra-\eta(Y)\eta_a(X)-\la X,\phi_a\phi Y\ra+\eta_a(X)\eta(Y)
=0.
\]
This gives Statement (a).

Let $\{e_1,\cdots,e_{4m-1}\}$ be an orthonormal basis on $T_xM$, $x\in M$. Then
it follows from $\trace(JJ_a)=0$ that 
\begin{align*}
0=\sum_j\la JJ_ae_j,e_j\ra+\la JJ_aN,N\ra=\sum_j\la\theta_ae_j,e_j\ra-\eta(\xi_a)=\trace(\theta_a)-\eta(\xi_a).
\end{align*}
This gives Statement (b).

Statements (c)--(g) can be obtained by direct calculations as below:
\begin{align*}
\theta^2_aX
	=&(\phi\phi_a-\xi\otimes\eta_a)(\phi_a\phi-\xi_a\otimes\eta)X=\phi\phi^2_a\phi X+\eta(X)\xi	\\
	=&-\phi^2X+\eta(X)\xi+\eta_a(\phi X)\phi\xi_a=X-\la\phi\xi_a,X\ra\xi_a;	\\
\theta_a\xi
	=&(\phi_a\phi-\xi_a\otimes\eta)\xi=-\xi_a; \\ 
\theta_a\xi_a
	=&(\phi\phi_a-\xi\otimes\eta_a)\xi_a=-\xi;		\\	
\theta_a\phi\xi_a
	=&(\phi_a\phi-\xi_a\otimes\eta)\phi\xi_a=\phi_a(-\xi_a+\eta(\xi_a)\xi)=\eta(\xi_a)\phi\xi_a;\\	
\theta_a\xi_{a+1}=&(\phi\phi_a-\xi\otimes\eta_a)\xi_{a+1}=\phi\phi_a\xi_{a+1}=\phi\xi_{a+2};\\
\theta_{a+1}\xi_a=&(\phi\phi_{a+1}-\xi\otimes\eta_{a+1})\xi_a=\phi\phi_{a+1}\xi_a=-\xi_{a+2}.
\end{align*}

\end{proof}
For each $x\in M$, we define a subspace $\mathcal H^\perp$ of $T_xM$ by
\[
\mathcal H^\perp:=\vspan\{\xi,\xi_1,\xi_2,\xi_3,\phi\xi_1,\phi\xi_2,\phi\xi_3\}.
\]

Let $\mathcal H$ be the orthogonal complement of $\mathbb H\mathbb C\xi$ in $T_xG_2(\mathbb C_{m+2})$.
Then $T_xM=\mathcal H\oplus\mathcal H^\perp$.
From the above identities, we see that $\mathcal H$ is 
invariant under $\phi$, $\phi_a$ and $\theta_a$.
It follows from Lemma~\ref{lem:theta}(c) that ${\theta_a}_{|{\mathcal H}}$ has two possible eigenvalues: $1$ and $-1$.

Let $\mathcal H_a(\varepsilon)$ be the eigenspace of ${\theta_a}_{|{\mathcal H}}$ corresponding to the eigenvalue $\varepsilon\in\{\pm1\}$.
Since $\theta_a\phi X=-\phi_aX=\phi\theta_a X$, for $X\in\mathcal H$, 
$\mathcal H_a(\varepsilon)$ is $\phi$-invariant and so it is of even dimension.
Moreover, since $\theta_a\theta_bX=-\phi_a\phi_bX=\phi_b\phi_aX=-\theta_b\theta_a X$, for $a\neq b$, $X\in\mathcal H$,  
and each ${\theta_a}_{|\mathcal H}$ is an automorphism in $\mathcal H$,
we see that $\theta_b\mathcal H_a(\varepsilon)=\mathcal H_a(-\varepsilon)$. Hence, each ${\theta_a}_{|\mathcal H}$ has exactly two eigenvalues $\pm1$ and $\dim H_a(1)=\dim H_a(-1)$ is even.

Further, for $X,Y\in\mathcal H_a(\varepsilon)$, $b\neq a$,
since $\theta_bX\in\mathcal H_a(-\varepsilon)$ and $\phi Y\in\mathcal H_a(\varepsilon)$, 
we have $\la\phi_b X,Y\ra=\la\theta_b X,\phi Y\ra=0$, that is, 
$\phi_b\mathcal H_a(\varepsilon)=\mathcal H_a(-\varepsilon)$.
We summarize these observations as below.

\begin{lemma}\label{lem:theta_2}
Let $\mathcal H_a(\varepsilon)$ be the eigenspace corresponds to eigenvalue $\varepsilon$ of 
${\theta_a}_{|\mathcal H}$.
Then 
\begin{enumerate}
\item[(a)] ${\theta_a}_{|\mathcal H}$ has two eigenvalues $\varepsilon=\pm1$,
\item[(b)] $\phi\mathcal H_a(\varepsilon)=\mathcal H_a(\varepsilon)$,
\item[(c)] $\theta_b\mathcal H_a(\varepsilon)=\mathcal H_a(-\varepsilon)$, for $a\neq b$,
\item[(d)] $\dim \mathcal H_a(1)=\dim \mathcal H_a(-1)$ is even, 
\item[(e)] $\phi_b\mathcal H_a(\varepsilon)=\mathcal H_a(-\varepsilon)$, for $a\neq b$.
\end{enumerate}
\end{lemma}


By the properties of $\theta_a$, we have 
\begin{lemma}
\begin{enumerate}
\item[(a)] $\xi\in\mathfrak D^\perp$ if and only if $\dim\mathcal H^\perp=3$.
\item[(b)] If $\xi\in\mathcal D$ then $\xi,\xi_1,\xi_2,\xi_3,\phi\xi_1,\phi\xi_2,\phi\xi_3$ are orthonormal.
\item[(c)] $\xi\notin\mathfrak D^\perp$ if and only if $\dim\mathcal H^\perp=7$.
\end{enumerate}
\end{lemma}
\begin{proof}
By Lemma~\ref{lem:theta_2}(d), $\mathcal H^\perp$ is of odd dimension. 
Since $\xi_1,\xi_2,\xi_3$ are orthonormal, we obtain $\dim\mathcal H^\perp\in\{3,5,7\}$
and Statement (a).

If $\xi\in\mathfrak D$ then 
\begin{align*}
0=\la\theta_a\phi\xi_a,\phi\xi_{a+1}\ra=\la\phi\xi_a,\theta_a\phi\xi_{a+1}\ra=\la\phi\xi_a,-\xi_{a+2}\ra. 
\end{align*}
Similarly, we also have $\la\phi\xi_a\xi_{a+1}\ra=0$ and we obtained Statement (b).

In view of Statement (b), we only have to verify the case: $\xi\notin\mathfrak D$ and $\xi\notin\mathfrak D^\perp$. 
We select an appropriate canonical local basis $\{J_1,J_2,J_3\}$ $\mathfrak J_{|M}$
such that $0<\eta(\xi_1)<1$, $\eta(\xi_2)=\eta(\xi_3)=0$.
It follows that $\la\xi_2,\phi\xi_3\ra=\eta(\xi_1)\neq0$ and so 
we have the following orthogonal eigenvectors of $\theta_1$: 
\[
\phi\xi_1,
\xi\pm\xi_1,\
\xi_2\pm\phi\xi_3,\
\xi_3\pm\phi\xi_2.
\]
This gives Statement (c).  
\end{proof}

\section{Semi-parallel real hypersurfaces in $G_2(\mathbb C_{m+2})$}

Recall that a tensor field $F$ of type $(1,s)$ of a Riemannian manifold $M$ is said to be \emph{semi-parallel} 
if $R\cdot F=0$, that is, 
\begin{align*}
(R(X,Y)F)(X_1,\cdots,X_s)=
&R(X,Y)F(X_1,\cdots,X_s)	\\
&-\sum_{i=1}^sF(X_1,\cdots,R(X,Y)X_i,\cdots,X_s)=0.
\end{align*}
A real hypersurface $M$ in $G_2(\mathbb C_{m+2})$ is said to be \emph{semi-parallel} if 
$R\cdot A=0$. 

Throughout  this section, we suppose $M$ is a semi-parallel real hypersurface in $G_2(\mathbb C_{m+2})$, $m\geq3$,
and we use the following notations:
\[
\alpha=\la A\xi,\xi\ra,\quad  
\alpha_a=\la A\xi_a,\xi_a\ra, \quad 
u_a=\eta_a(\xi).
\] 

Let $Y,Z\in T_xM$, $x\in M$. 
It follows from $\la (R(\xi,Y)A)Z,\xi\ra=0$ that 
\begin{align}\label{eqn:130}
&\alpha\la AY,AZ\ra+\{1-||A\xi||^2\}\la Y,AZ\ra-\alpha\la Y,Z\ra
		-\la A^2\xi,Z\ra\la A\xi,Y\ra+\la A^2\xi,Y\ra\la A\xi,Z\ra	\nonumber\\
&-\la A\xi,Z\ra\la \xi,Y\ra+\la A\xi,Y\ra\la \xi,Z\ra							
		+\sum_{a=1}^3\{3\la A\phi\xi_a,Z\ra\la\phi\xi_a,Y\ra-\la\phi_a Y,Z\ra\la A\xi,\phi\xi_a\ra \nonumber\\
&-\la\phi_a A\xi,Y\ra\la\phi\xi_a,Z\ra	
		-2\la\phi_a A\xi,Z\ra\la\phi\xi_a,Y\ra-u_a\la A\theta_a Y,Z\ra-\la A\xi_a,Z\ra\la\xi_a,Y\ra
																																								\nonumber\\
&	+\la\theta_aY,Z\ra\la A\xi_a,\xi\ra-\la \theta_aA\xi,Y\ra\la\xi_a,Z\ra\}=0;
\end{align}
By switching $Y$ and $Z$ in this equation, and then subtracting the obtained equation from (\ref{eqn:130}), we obtain
\begin{align}\label{eqn:140}
&-2\la A^2\xi,Z\ra\la A\xi,Y\ra+2\la A^2\xi,Y\ra\la A\xi,Z\ra
		-2\la A\xi,Z\ra\la \xi,Y\ra	+2\la A\xi,Y\ra\la \xi,Z\ra						\nonumber\\
&+\sum_{a=1}^3\{3\la A\phi\xi_a,Z\ra\la\phi\xi_a,Y\ra-3\la A\phi\xi_a,Y\ra\la\phi\xi_a,Z\ra
		-2\la\phi_a Y,Z\ra\la A\xi,\phi\xi_a\ra \nonumber\\
&-\la\phi_a A\xi,Z\ra\la\phi\xi_a,Y\ra+\la\phi_a A\xi,Y\ra\la\phi\xi_a,Z\ra	
		-\la A\xi_a,Z\ra\la\xi_a,Y\ra+\la A\xi_a,Y\ra\la\xi_a,Z\ra	\nonumber\\
&+\la\theta_aA\xi,Z\ra\la\xi_a,Y\ra-\la \theta_aA\xi,Y\ra\la\xi_a,Z\ra							
		-u_a\la A\theta_a Y-\theta_aAY,Z\ra\}=0.
\end{align}
Let $\{e_1,\cdots,e_{4m-1}\}$ be an orthonormal basis on $T_xM$, $x\in M$.  
Then it follows from 
\[
\sum_j\{\la (R(e_j,Y)A)Z,e_j\ra-\la (R(e_j,Z)A)Y,e_j\ra\}=0
\]
that
\begin{align}\label{eqn:120}
&-3\la A\xi,Z\ra\la \xi,Y\ra+3\la A\xi,Y\ra\la \xi,Z\ra
			+\sum_{a=1}^3\{-3\la A\xi_a,Z\ra\la\xi_a,Y\ra+3\la A\xi_a,Y\ra\la\xi_a,Z\ra	\nonumber \\
  &+\la A\phi\xi_a,Z\ra\la\phi\xi_a,Y\ra-\la A\phi\xi_a,Y\ra\la\phi\xi_a,Z\ra+u_a\la A\theta_a Y-\theta_a AY,Z\ra\}=0.
\end{align}
Also, from 
$\la (R(Z,Y)A)\xi,\xi\ra=0$, we have
\begin{align}\label{eqn:160}
&-\la A^2\xi,Z\ra\la A\xi,Y\ra+\la A^2\xi,Y\ra\la A\xi,Z\ra
		-\la A\xi,Z\ra\la \xi,Y\ra+\la A\xi,Y\ra\la \xi,Z\ra							\nonumber\\
&+\sum_{a=1}^3\{-\la\phi_a A\xi,Z\ra\la\phi\xi_a,Y\ra+\la\phi_a A\xi,Y\ra\la\phi\xi_a,Z\ra
		-2\la\phi_a Y,Z\ra\la A\xi,\phi\xi_a\ra 															 			\nonumber\\
&+\la\theta_aA\xi,Z\ra\la\xi_a,Y\ra-\la \theta_aA\xi,Y\ra\la\xi_a,Z\ra\}=0.
\end{align}
$(\ref{eqn:140})-2\times(\ref{eqn:160})$:
\begin{align}\label{eqn:200}
&\sum_{a=1}^3\{3\la A\phi\xi_a,Z\ra\la\phi\xi_a,Y\ra-3\la A\phi\xi_a,Y\ra\la\phi\xi_a,Z\ra
		+2\la\phi_a Y,Z\ra\la A\xi,\phi\xi_a\ra 																		\nonumber\\
&+\la\phi_a A\xi,Z\ra\la\phi\xi_a,Y\ra-\la\phi_a A\xi,Y\ra\la\phi\xi_a,Z\ra	
		-\la A\xi_a,Z\ra\la\xi_a,Y\ra+\la A\xi_a,Y\ra\la\xi_a,Z\ra									\nonumber\\	
&-\la\theta_aA\xi,Z\ra\la\xi_a,Y\ra+\la \theta_aA\xi,Y\ra\la\xi_a,Z\ra
		-u_a\la A\theta_a Y-\theta_aAY,Z\ra\}=0.
\end{align}
$(\ref{eqn:140})+(\ref{eqn:120})-(\ref{eqn:160})$:
\begin{align}\label{eqn:180}
&-\la A^2\xi,Z\ra\la A\xi,Y\ra+\la A^2\xi,Y\ra\la A\xi,Z\ra
		-4\la A\xi,Z\ra\la \xi,Y\ra+4\la A\xi,Y\ra\la \xi,Z\ra							\nonumber\\
&+\sum_{a=1}^3\{4\la A\phi\xi_a,Z\ra\la\phi\xi_a,Y\ra-4\la A\phi\xi_a,Y\ra\la\phi\xi_a,Z\ra	\nonumber\\
&-4\la A\xi_a,Z\ra\la\xi_a,Y\ra+4\la A\xi_a,Y\ra\la\xi_a,Z\ra\}=0.	
\end{align}

Consider two orthonormal principal vectors $Y_j$ and $Y_k$, corresponding to principal curvatures $\lambda_j$ and $\lambda_k$ respectively. Then from 
$\la(R(Y_k,Y_j)A)Y_j,Y_k\ra=0$, we have
\begin{align}\label{eqn:curvature}
&(\lambda_j-\lambda_k)\{\lambda_j\lambda_k+1+3\la Y_k,\phi Y_j\ra^2\}		\nonumber\\
 &+(\lambda_j-\lambda_k)\sum_{a=1}^3\{3\la Y_k,\phi_aY_j\ra^2+\la\theta_aY_j,Y_j\ra\la\theta_aY_k,Y_k\ra
 -\la\theta_aY_k,Y_j\ra^2\}=0.
\end{align} 

Finally, from $\la (R(Z,Y)A)\xi_b,\xi_b\ra=0$, for $b\in\{1,2,3\}$, we have
\begin{align}\label{eqn:700}
&-\la A^2\xi_b,Z\ra\la A\xi_b,Y\ra+\la A^2\xi_b,Y\ra\la A\xi_b,Z\ra
 -\la A  \xi_b,Z\ra\la  \xi_b,Y\ra+\la A  \xi_b,Y\ra\la  \xi_b,Z\ra							\nonumber\\
&+\la\phi A\xi_b,Y\ra\la\phi\xi_b,Z\ra-\la\phi A\xi_b,Z\ra\la\phi\xi_b,Y\ra
		+2\la\phi Z,Y\ra\la A\xi_b,\phi\xi_b\ra																		\nonumber\\		
&+\sum_{a=1}^3\{\la\phi_a A\xi_b,Y\ra\la\phi_a\xi_b,Z\ra-\la\phi_a A\xi_b,Z\ra\la\phi_a\xi_b,Y\ra
		+2\la\phi_a Z,Y\ra\la A\xi_b,\phi_a\xi_b\ra																		\nonumber\\	
&+\la\theta_aA\xi_b,Y\ra\la\theta_a\xi_b,Z\ra
 -\la\theta_aA\xi_b,Z\ra\la\theta_a\xi_b,Y\ra\}=0.
\end{align}

The proof of Theorem~\ref{thm:main} is broken into three steps.
We shall show that these following three cases cannot occur.
\begin{center}
 $\xi\notin\mathfrak D$ and $\xi\notin\mathfrak D$;\quad 
$\xi\in\mathfrak D$; \quad $\xi\in\mathfrak D^\perp$.
\end{center}

\subsection{The case: $\xi\notin\mathfrak D$ and $\xi\notin\mathfrak D^\perp$}
Suppose $\xi\notin\mathfrak D$ and $\xi\notin\mathfrak D^\perp$ at a point $x\in M$.
Without loss of generality, we assume $0<u_1<1$, $u_2=u_3=0$.
\begin{lemma}	\label{lem:1}
$A\xi$, $A\xi_1\in\vspan\{\xi,\xi_1\}$.
\end{lemma}
\begin{proof}
We shall first prove that $A\xi$, $A\xi_1\in\mathcal H^\perp$.
By first putting $Z=\xi$, and next $\xi_1$  in (\ref{eqn:120}), we obtain
\begin{align}\label{eqn:300}
-3\alpha\la \xi,Y\ra+3\la A\xi,Y\ra+u_1\la A\xi,\theta_1Y\ra+4u_1\la A\xi_1,Y\ra	\nonumber\\
+\sum_{a=1}^3\{-3\la A\xi,\xi_a\ra\la\xi_a,Y\ra+\la A\xi,\phi\xi_a\ra\la\phi\xi_a,Y\ra\}=0;
\end{align}
\begin{align}\label{eqn:310}
-3\la A\xi,\xi_1\ra\la \xi,Y\ra+4u_1\la A\xi,Y\ra+3\la A\xi_1,Y\ra+u_1\la A\xi_1,\theta_1Y\ra	\nonumber\\
+\sum_{a=1}^3\{-3\la A\xi_1,\xi_a\ra\la\xi_a,Y\ra+\la A\xi_1,\phi\xi_a\ra\la\phi\xi_a,Y\ra\}=0.
\end{align}
By choosing $Y\in\mathcal H_1(1)$ in the above two equations, we have
\begin{align*}
(3+u_1)\la A\xi,Y\ra+4u_1\la A\xi_1,Y\ra =0\\
4u_1\la A\xi,Y\ra+(3+u_1)\la A\xi_1,Y\ra =0.
\end{align*}
It follows from these equations that $A\xi$, $A\xi_1\perp\mathcal H_1(1)$.

Next, if we first put $Y\in\mathcal H_1(-1)$ in (\ref{eqn:300}), followed by $Y\in\mathcal H_1(-1)$ and $Z=\xi_1$ in (\ref{eqn:200}), then
\begin{align*}
(3-u_1)\la A\xi,Y\ra+4u_1\la A\xi_1,Y\ra =0\\
(1+u_1)\{-\la A\xi,Y\ra+\la A\xi_1,Y\ra\} =0.
\end{align*}
Solving these equations, gives $A\xi$, $A\xi_1\perp\mathcal H_1(-1)$ and we conclude that  
\begin{align}\label{eqn:324}
A\xi, A\xi_1\perp\mathcal H.
\end{align}

Secondly, we shall show that $A\xi, A\xi_1\perp\phi\xi_b$, $\xi_c$, for $b\in\{1,2,3\}$, $c\in\{2,3\}$.
Let $Z\in\mathcal H$ and $Y=\phi Z$ in (\ref{eqn:160}). Then $\sum^3_{a=1}\la\theta_a Z,Z\ra\la A\xi,\phi\xi_a\ra=0$.
In particular, if we choose a unit vector $Z\in\mathcal H_b(1)$ then by Lemma~\ref{lem:theta_2}(c), 
$\theta_aZ\in\mathcal H_b(-1)$, which implies that 
$\la\theta_aZ,Z\ra=0$, for $a\neq b$ and so 
\begin{align}\label{eqn:325}
\la A\xi,\phi\xi_b\ra=0, \quad b\in\{1,2,3\}.
\end{align} 
By putting $Y=\xi_c$ in (\ref{eqn:300}), with the help of (\ref{eqn:325}), we have 
$4u_1\la A\xi_1,\xi_c\ra=0$ and so
\begin{align}\label{eqn:326}
\la A\xi_1,\xi_c\ra=0, \quad c\in\{2,3\}.
\end{align}
Next, by using the above two results, after putting $Y=\phi\xi_b$ in (\ref{eqn:310}), we get
\[
(u_1^2+4)\la A\xi_1,\phi\xi_1\ra=4\la A\xi_1,\phi\xi_2\ra=4\la A\xi_1,\phi\xi_3\ra=0
\]
these mean that
\begin{align} \label{eqn:327}
\la A\xi_1,\phi\xi_b\ra=0, \quad b\in\{1,2,3\}.
\end{align}
Similarly, with the helps of (\ref{eqn:325})--(\ref{eqn:327}), it follows from (\ref{eqn:310}) 
that $\la A\xi,\xi_2\ra=\la A\xi,\xi_3\ra=0$.
From this result, together with (\ref{eqn:324})--(\ref{eqn:327}), gives the lemma.
\end{proof}

We define a unit vector $U:=(\xi_1-u_1\xi)/{\sqrt{1-u_1^2}}$.
Note that $\{\xi, U\}$ is an orthonormal basis for $\vspan\{\xi,\xi_1\}$.
From Lemma~\ref{lem:1}, and (\ref{eqn:300}), we have 
\begin{align}\label{eqn:Axi}
A\xi  =&\alpha\xi+\rho U; \quad AU=\rho\xi+\sigma U, \\
A\xi_1=&\alpha\xi_1+\rho\left\{\sqrt{1-u_1^2}\xi-u_1U\right\} 
\label{eqn:Axi1}.
\end{align}
Further, by putting $Z=\xi$ and $Y=U$ in (\ref{eqn:160}), we obtain
\begin{align}\label{eqn:beta2}
\rho(\alpha\sigma-\rho^2)=0.
\end{align}
By virtue of (\ref{eqn:Axi}), if we choose $Z\in\mathcal H$, then (\ref{eqn:180}) gives 
\begin{align*}
\sum_{a=1}^3\{\la A\phi\xi_a,Z\ra \phi\xi_a-\la A\xi_a,Z\ra\xi_a\}=0.	
\end{align*}
Since $\{\xi_1,\xi_2,\xi_3,\phi\xi_1,\phi\xi_2,\phi\xi_3\}$ is linearly independent, 
$\la A\phi\xi_a,Z\ra=\la A\xi_a,Z\ra=0$ and so
$A\xi_a$, $A\phi\xi_a$ $\perp$ $\mathcal H$.
Hence we obtain 
\begin{align}\label{eqn:AH=Ha}
A\mathcal H\subset\mathcal H.
\end{align}
By making use of (\ref{eqn:200}), (\ref{eqn:Axi}) and (\ref{eqn:AH=Ha}), we obtain
\begin{align}\label{eqn:AH=Hb}
A\theta_1Y=\theta_1AY, \quad Y\in\mathcal H. 
\end{align}

\begin{lemma}\label{lem:Ei}
Either all principal curvatures vanish or none of them is zero. 
\end{lemma}
\begin{proof}
We first show that $\phi\xi_1$ (and so is $\phi U$) is principal.
By putting $Y=\phi\xi_1$ in (\ref{eqn:180}), with the help of (\ref{eqn:Axi})--(\ref{eqn:beta2}), we obtain
\[
(1-u_1^2)\la A\phi\xi_1,Z\ra
+\sum_{a=1}^3\{-\la A\phi\xi_1,\phi\xi_a\ra\la\phi\xi_a,Z\ra+\la A\phi\xi_1,\xi_a\ra\la\xi_a,Z\ra\}=0.
\]
This equation implies that $A\phi\xi_1$ is perpendicular to the vectors 
\[
\xi_2-u_1\phi\xi_3,\
(2-u_1^2)\xi_2-u_1\phi\xi_3;\
\xi_3+u_1\phi\xi_2,\
(2-u_1^2)\xi_3+u_1\phi\xi_2
\]
which deduces that $A\phi\xi_1$ $\perp$ $\xi_2,\xi_3,\phi\xi_2,\phi\xi_3$. This fact, together with 
Lemma~\ref{lem:1} and (\ref{eqn:AH=Ha}), gives $A\phi\xi_1=\delta_0\phi\xi_1$. 

Since $\mathcal H$ is invariant under $A$ and $\theta_1$; and by 
(\ref{eqn:AH=Hb}),
we can construct orthonormal bases $\{X_1,\cdots,X_{4m-8}\}$ for $\mathcal H$ 
and $\{E_0=\phi U,E_1,\cdots,E_6\}$ for $\mathcal H^\perp$
such that 
\begin{align*}
AX_r=&\lambda_r X_r, \quad r\in\{1,\cdots,4m-8\} \nonumber\\
\theta_1X_r=&\left\{\begin{array}{rl}
X_r, & r\in\{1,\cdots,2m-4\}; \\
-X_r, & r\in\{2m-3,\cdots,4m-8\}.\end{array}\right.	\nonumber\\
AE_i=&\delta_iE_i, \quad i\in\{0,\cdots,6\}.
\end{align*}

By putting $Y_j=E_i$ and $Y_k=X_r$ 
in (\ref{eqn:curvature}),
since $\la\theta_bX_r,X_r\ra=0$, for $b\in\{2,3\}$,
we obtain
\begin{align}\label{eqn:a}
(\delta_i-\lambda_r)\{\delta_i\lambda_r+1+\la\theta_1X_r,X_r\ra\la\theta_1E_i,E_i\ra\}=0.
\end{align}

We consider two cases: $\delta_0=0$; and $\delta_0\neq0$.

\medskip
\textbf{Case 1.} $\delta_0=0$.

We set $i=0$ in (\ref{eqn:a}) to get
$
0=\lambda_r(1\pm\la\theta_1\phi U,\phi U\ra)=\lambda_r(1\pm u_1).
$ 
Hence $\lambda_r=0$, for $r=\{1,\cdots,4m-8\}$ and 
(\ref{eqn:a}) reduces to 
\[
\delta_i(1+\la\theta_1X_r,X_r\ra\la\theta_1E_i,E_i\ra)=0. 
\]
Since $\la\theta_1X_1,X_1\ra=1$ and $\la\theta_1X_{4m-8},X_{4m-8}\ra=-1$, 
all $\delta_i=0$ and so all principal curvatures are zero.

\medskip
\textbf{Case 2.} $\delta_0\neq0$.

We first claim that all $\lambda_r\neq0$.
Suppose to the contrary that $\lambda_s=0$, for some $s\in\{1,$ $\cdots,4m-8\}$. 
We set $i=0$ and $r=s$ in (\ref{eqn:a}) to get
$0=1\pm\la\theta_1\phi U,\phi U\ra=1\pm u_1$. 
This is a contradiction. We conclude that $\lambda_r\neq0$, for $r\in\{1,\cdots,4m-8\}$.

Next, we claim that all $\delta_i\neq0$. For otherwise, we can set  
$r=1$ and $r=4m-8$ respectively in (\ref{eqn:a}) to obtain a contradiction.
This completes the proof. 
\end{proof}
\begin{lemma}\label{lem:xi=principal}
$A\xi=\alpha\xi$ and  $A\xi_a=\alpha_a\xi_a$, for  $a\in\{1,2,3\}$.
\end{lemma}
\begin{proof}
Suppose $\xi$ is not principal or $\rho\neq0$. 
In view of (\ref{eqn:beta2}), $\alpha\sigma=\rho^2$ and so $\alpha+\sigma\neq0$. Further, we can verify that 
$A(\rho\xi-\alpha U)=0$ and $A(\alpha\xi+\rho U)=(\alpha+\sigma)(\alpha\xi+\rho U)$.
But this contradicts Lemma~\ref{lem:Ei}, hence we 
conclude that $A\xi=\alpha\xi$ and $\rho=0$.
From (\ref{eqn:Axi1}), we can see that $\xi_1$ is also a principal vector.

Next, fixed $b\in\{2,3\}$.
Let $Y$ be a unit vector in $\mathcal H$ and $Z=\phi Y$ in (\ref{eqn:700}). Then 
\[
-\la A\xi_b,\phi\xi_b\ra+\la\theta_1Y,Y\ra\la A\xi_b,\phi_1\xi_b\ra=0.
\]
By first putting $Y\in\mathcal H_1(-1)$, followed by $Y\in\mathcal H_1(1)$,
we obtain $\la A\xi_b,\phi\xi_b\ra=\la A\xi_b,\phi_1\xi_b\ra=0$, more precisely
$\la A\xi_2,\phi\xi_2\ra=\la A\xi_3,\phi\xi_3\ra=\la A\xi_2,\xi_3\ra=0$.
It follows that we may write
\begin{align}\label{eqn:Axi2}
A\xi_b=\alpha_b\xi_b+\rho_bU_b, \quad 
U_2:=\dfrac{\phi\xi_3-u_1\xi_2}{\sqrt{1-u_1^2}},\quad
U_3:=\dfrac{\phi\xi_2+u_1\xi_3}{\sqrt{1-u_1^2}}.
\end{align}
In view of the above equation, after putting $Z=U_b$ in (\ref{eqn:700}), we have
\begin{align*}
\rho_b\{\rho_bAU_b-\sigma_bA\xi_b\}=0,\quad (\sigma_b:=\la AU_b,U_b\ra).
\end{align*}
From this equation, we can see that $\xi_b$ is principal, for $b\in\{2,3\}$.
Indeed, if $\xi_b$ is not principal or equivalently $\rho_b\neq0$ then the above equation implies that 
\begin{align}\label{eqn:AU2}
AU_b=\rho_b\xi_b+\sigma_bU_b; \quad \alpha_b\sigma_b=\rho_b^2.
\end{align}
It follows from (\ref{eqn:Axi2}) and (\ref{eqn:AU2}) that  
$\rho_b\xi_b-\alpha_bU_b$ and $\alpha_b\xi_b+\rho_bU_b$ are principal vectors correspond to the principal curvatures $0$ and $\alpha_b+\sigma_b\neq0$ respectively.
This contradicts Lemma~\ref{lem:Ei},
hence $\xi_b$ is principal, for $b\in\{2,3\}$.
\end{proof}

\begin{lemma}\label{lem:case1}
Let $M$ be a semi-parallel real hypersurface in $G_2(\mathbb C_{m+2})$, $m\geq3$. 
Then either $\xi\in\mathfrak D$ or $\xi\in\mathcal D^\perp$ at each point $x\in M$. 
\end{lemma}
\begin{proof} 
Consider the open subset
\[
M_0:=\{x\in M:~ g(x):=u_1^2+u_2^2+u_3^2\notin\{0,1\}\}.
\]
Then by Lemma~\ref{lem:xi=principal}, we have $A\xi=\alpha\xi$ and 
$A\mathfrak D^\perp\subset\mathfrak D^\perp$ on $M_0$. 
In view of Theorem~\ref{thm:classification}--\ref{thm:B}, $M_0$ is an open part of a real hypersurface of type $A$ or $B$ and either $\xi\in\mathfrak D$ or $\xi\in\mathfrak D^\perp$ at each point in $M_0$. Hence $M_0$ is empty and this completes the proof.
\end{proof}

\subsection{The case: $\xi\in\mathcal D$}

Suppose $\xi\in\mathfrak D$ at each point $x\in M$. 
Then the vectors $\xi,\xi_1,\xi_2, \xi_3, \phi\xi_1, \phi\xi_2, \phi\xi_3$ are orthonormal and each $u_a=0$.
By using Lemma~\ref{lem:Xua}, we have
\[
0=\la(R(\phi\xi_1,\xi)A)Z,\phi\xi_1\ra=
\la A\xi,Z\ra+3\sum_{a=1}^3\la A\phi_a\phi\xi_1,Z\ra\la\phi_a\phi\xi_1,\xi\ra=4\la A\xi,Z\ra. 
\]
Hence $A\xi=0$. By Theorem~\ref{thm:lee-suh}, $M$ is an open part of a real hypersurface of type $B$.
This is a contradiction as $\alpha\neq0$ according to Theorem~\ref{thm:B}.
Hence we obtain the following lemma.

\begin{lemma}\label{lem:case2}
Let $M$ be a semi-parallel real hypersurface in $G_2(\mathbb C_{m+2})$, $m\geq3$. 
Then $\xi\in\mathcal D^\perp$ at each point $x\in M$. 
\end{lemma}

\subsection{The case: $\xi\in\mathfrak D^\perp$}
We suppose that $\xi\in\mathfrak D^\perp$ at each $x\in M$. Let $J_1\in\mathfrak J_x$ such that $J_1N=JN$.
Then we have
\[
\xi_1=\xi=-\theta_1\xi_1,\ 
\xi_2=\theta_1\xi_2=\phi\xi_3,\
\xi_3=\theta_1\xi_3=-\phi\xi_2,\
u_1=1,\
u_2=u_3=0
\]

\begin{lemma}\label{lem:hopf3}
$	A\xi=\alpha\xi$.
\end{lemma}
\begin{proof}
Fixed $b\in\{2,3\}$.
By putting $Y=\xi_b$ and $Z=\xi$ in (\ref{eqn:120}), we have $\la A\xi,\xi_b\ra=0$.
Using this fact, after putting $Y\in\mathcal H$ and $Z=\xi$ in (\ref{eqn:120}), we obtain
$\la A\xi,7Y+\theta_1Y\ra=0$. 
By using Lemma~\ref{lem:theta_2}(a),
 we have 
$A\xi$ $\perp$ $\mathcal H$. Hence we conclude that  
$A\xi=\alpha\xi$.
\end{proof}


Fixed $b\in\{2,3\}$.
By putting $Y\in\mathcal H$ and $Z=\xi_b$ in (\ref{eqn:120}), we obtain
$\la A\xi_b,Y+\theta_1Y\ra=0$, which implies that $A\xi_b$ $\perp$ $\mathcal H_1(1)$.
Next, by putting $Y\in\mathcal H_1(1)$ and $Z\in\mathcal H_1(-1)$ in (\ref{eqn:120}), we have $2\la AY,Z\ra=0$.
This implies that
\begin{align}\label{eqn:AH1+a}
A\mathcal H_1(1)\subset\mathcal H_1(1).
\end{align}

Let $Y\in\mathcal H_1(1)$ be a unit vector and $Z=\phi Y$ in (\ref{eqn:700}).
Then 
\begin{align*}
0=&-\la A\xi_b,\phi\xi_b\ra+\sum_{a=1}^3\la\phi_a\phi Y,Y\ra\la A\xi_b,\phi_a\xi_b\ra	\\
 =&-\la A\xi_b,\phi\xi_b\ra+\sum_{a=1}^3\la\theta_a Y,Y\ra\la A\xi_b,\phi_a\xi_b\ra			\\
 =&-\la A\xi_b,\phi\xi_b\ra+\la A\xi_b,\phi_1\xi_b\ra=-2\la A\xi_b,\phi\xi_b\ra.			
\end{align*}
Hence, $\la A\xi_3,\xi_2\ra=0$ and so we obtain 
\begin{align}\label{eqn:AH1+b}
A\xi_b-\alpha_b\xi_b\in\mathcal H_1(-1),  \quad b\in\{2,3\}.
\end{align}
By substituting $Z=\xi_b$ in (\ref{eqn:700}), we obtain
\begin{align*}
0&=-||A\xi_b||^2A\xi_b+\alpha_bA^2\xi_b-\alpha_b\xi_b+A\xi_b+\theta_1A\xi_b-\alpha_b\theta_1\xi_b.	
\end{align*}
By (\ref{eqn:AH1+b}), we have
$\theta_1(A\xi_b-\alpha_b\xi_b)=-(A\xi_b-\alpha_b\xi_b)$.
Hence we obtain 
\begin{align}\label{eqn:AH1+c}
\alpha_b A^2\xi_b-||A\xi_b||^2A\xi_b=0,  \quad b\in\{2,3\}.
\end{align}
\begin{lemma} \label{lem:AX}
Suppose $\xi_b$ is not principal, for some $b\in\{2,3\}$. 
Let $\rho_b=||\phi_bA\xi_b||$ and $U_b=-\rho_b^{-1}\phi_b^2A\xi_b$. Then 
\begin{align}\label{eqn:AUb3}
AU_b=\rho_b\xi_b+\sigma_bU_b, \quad \alpha_b\sigma_b=\rho_b^2.
\end{align}

If $X\in\mathcal H_1(1)$ is a unit vector with $AX=\lambda X$ then either $\lambda=0$ or $\lambda=\alpha$ $(\neq0)$.
Further we have
\begin{enumerate}
\item[(a)] 
if $\lambda=\alpha$ then $0\geq 2\rho_b^2-\alpha_b^2\sum_{a=2}^3\la\theta_a U_b,X\ra^2$; and
\item[(b)]
if $\lambda=0$       then $0\geq 2\alpha_b^2-\rho_b^2\sum_{a=2}^3\la\theta_a U_b,X\ra^2$. 
\end{enumerate}
\end{lemma}
\begin{proof}
The equation (\ref{eqn:AUb3}) is an immediate consequence of (\ref{eqn:AH1+c}).
If $\alpha=0$ then we set $Y\in\mathcal H_1(-1)$ in (\ref{eqn:130}) to get $AY=0$ and so $A\mathcal H_1(-1)\subset\mathcal H_1(-1)$. But this contradicts our assumption and (\ref{eqn:AH1+b}).
Hence, we have $\alpha\neq0$.

Let $X\in\mathcal H_1(1)$ be a unit vector with $AX=\lambda X$,
the existence of such $X$ is ensured by (\ref{eqn:AH1+a}). 
By putting $Y=Z=X$ in (\ref{eqn:130}), we have
$\lambda(\lambda-\alpha)=0$. 
On the other hand, we can verify that 
$AE_1=0$ and $AE_2=(\alpha_b+\sigma_b)E_2$,
where 
\[
E_1:=\frac{\rho_b\xi_b-\alpha_b U_b}{\sqrt{\rho_b^2+\alpha_b^2}}; \quad 
E_2:=\frac{\alpha_b\xi+\rho_b U_b}{\sqrt{\rho_b^2+\alpha_b^2}}.
\]

Suppose $\lambda=\alpha$.
By setting $Y_k=E_1$ and $Y_j=X$ in (\ref{eqn:curvature}), we obtain
\begin{align*}
0&=1+\la\theta_1E_1,E_1\ra+\sum^3_{a=1}\{3\la\phi_aX,E_1\ra^2-\la\theta_aE_1,X\ra^2\}		\\
0&\geq1+\la\theta_1E_1,E_1\ra-\sum^3_{a=2}\la\theta_aE_1,X\ra^2		
 		=1+\frac{\rho_b^2-\alpha_b^2}{\rho_b^2+\alpha_b^2}
 				-\sum^3_{a=2}\frac{\alpha_b^2\la\theta_aU_b,X\ra^2}{\rho_b^2+\alpha_b^2}	\\
 &\geq2\rho_b^2-\alpha_b^2\sum_{a=2}^3\la\theta_a U_b,X\ra^2.
\end{align*} 
Similarly, when $\lambda=0$, we set $Y_k=E_2$ and $Y_j=X$ in (\ref{eqn:curvature}) to get
\begin{align*}
0&\geq1+\la\theta_1E_2,E_2\ra-\sum^3_{a=2}\la\theta_aE_2,X\ra^2		
 		=1+\frac{\alpha_b^2-\rho_b^2}{\rho_b^2+\alpha_b^2}
 				-\sum^3_{a=2}\frac{\rho_b^2\la\theta_aU_b,X\ra^2}{\rho_b^2+\alpha_b^2}	\\
 &\geq2\alpha_b^2-\rho_b^2\sum_{a=2}^3\la\theta_a U_b,X\ra^2.
\end{align*} 
\end{proof}
\begin{lemma}\label{lem:hopf3b}
$A\xi_b=\alpha_b\xi_b$, for $b\in\{2,3\}$.
\end{lemma}
\begin{proof}
Suppose to the contrary that $\xi_b$ is not principal.
We consider the following two cases.

\medskip
\textbf{Case 1.} $AX=\alpha X$ or $AX=0$, for all $X\in\mathcal H_1(1)$.

Since $\dim \mathcal H_1(1)=2m-2\geq4$, there is a unit vector $X\in\mathcal H_1(1)$, which is perpendicular to $\theta_2U_b$ and $\theta_3U_b$. Then Lemma~\ref{lem:AX}(a)--(b) imply  that either $0\geq2\rho_b^2$
or $0\geq2\alpha_b^2$. However, this contradicts the fact that $\alpha_b\sigma_b=\rho_b^2\neq0$.
Hence this case cannot occur.
	
\medskip
\textbf{Case 2.} $AX_1=\alpha X_1$ and $AX_2=0$, for some unit vectors $X_1$, $X_2\in\mathcal H_1(1)$.

The two inequalities in Lemma~\ref{lem:AX} imply that 
\begin{align*}
0\geq \alpha_b^2\sum_{a=2}^3\{1-\la\theta_a U_b,X_1\ra^2\}+\rho_b^2\sum_{a=2}^3\{1-\la\theta_a U_b,X_2\ra^2\}.
\end{align*}
Since $\theta_aU_b$ and $X_1$ are unit vectors, by Cauchy-Schwarz Inequality,
$1-\la\theta_aU_b,X_1\ra^2\geq0$ and equality holds if and only if $X_1=\pm\theta_aU_b$.
Further, since $\theta_2\theta_3U_b=-\theta_3\theta_2U_b$, we have $\theta_2U_b$ and $\theta_3U_b$ are orthonormal.
Hence
$1-\la\theta_2U_b,X_1\ra^2$ and $1-\la\theta_3U_b,X_1\ra^2$ cannot be both zero.
Hence we conclude that $\sum_{a=2}^3\{1-\la\theta_a U_b,X_1\ra^2\}>0$.
Similarly, we have $\sum_{a=2}^3\{1-\la\theta_a U_b,X_2\ra^2\}>0$.
It follows that $\alpha_b=\rho_b=0$. This is a contradiction.

In view of these cases, we conclude that $\xi_b$ is principal for $b\in\{2,3\}$.
\end{proof}

\subsection{Proof of Theorem~\ref{thm:main}}
We are in a position to prove Theorem~\ref{thm:main}.
We have showed that if a real hypersurface $M$ in $G_2(\mathbb C_{m+2})$ is semi-parallel then $A\xi=\alpha\xi$ and $A\mathfrak D^\perp\subset\mathfrak D^\perp$; and 
$\xi\in\mathfrak D^\perp$ at each point in $M$.
According to Theorem~\ref{thm:classification}, $M$ is an open part of a real hypersurface of type $A$.
Now we follow the notations in Theorem~\ref{thm:A}.
If we put $Y_k=\xi_2$ and $Y_j\in T_\mu=\mathcal H_1(1)$ in (\ref{eqn:curvature}), then
\[
0=1+\sum_{a=1}^3\la\theta_aY_j,Y_j\ra\la\theta_a\xi_2,\xi_2\ra=2
\]
which is impossible.
This completes the proof of Theorem~\ref{thm:main}.

\section{Recurrent real hypersurfaces in $G_2(\mathbb C_{m+2})$}
In this section, we shall show that there are no recurrent real hypersurfaces in $G_2(\mathbb C_{m+2})$.
We begin with the following result.

\begin{theorem}\label{thm:recurrent}
Let $M$ be a connected Riemannian manifold and let $F$ be a symmetric endomorphism on $TM$.
If $F$ is recurrent then $F$ is semi-parallel. 
\end{theorem}
\begin{proof}
Suppose $F$ is recurrent, that is $(\nabla_X F)Y=\omega(X)FY$, for all $X,Y\in TM$, where $\omega$ is a $1$-form on $M$ and $\nabla$ is the levi-Civita connection on $M$.
Let $\mathcal U$ be the maximal open dense subset of $M$ such that the multiplicities of the eigenvalue functions of $F$ are constant in each component of $\mathcal U$.
It is trivial if $F=0$ on $\mathcal U$. Consider a point $x\in\mathcal U$ such that $F\neq0$. Then there is an eigenvalue $\lambda$ of $F$, which is nowhere zero in a neighborhood $\mathcal U_1\subset\mathcal U$ of $x$. 
Let $Y$ be a unit eigenvector field of $F$ on $\mathcal U_1$ corresponding to $\lambda$. Then
\[
X\lambda=\la(\nabla_XF)Y,Y\ra=\lambda\omega(X)
\]
for all vector field $X$ on $\mathcal U_1$, or equivalently $d\lambda=\lambda\omega$. Hence
\[
0=d^2\lambda=d\lambda\wedge\omega+\lambda d\omega=\lambda d\omega.
\]
This means that $d\omega=0$ at $x$ and so $(\nabla_X\omega)Y=(\nabla_Y\omega)X$, for all $X,Y\in T_xM$. 

Denote by $\nabla_{X,Y}F$ the second order covariant derivative of $F$, that is,
\[
(\nabla_{X,Y}F)Z=\nabla_X\{(\nabla_YF)Z\}-(\nabla_YF)\nabla_XZ-(\nabla_{\nabla_XY}F)Z
\]
for all vector fields $X,Y,Z$ tangent to $M$.
It follows that 
\[
(\nabla_{X,Y}F)Z=\{(\nabla_X\omega)Y\}FZ+\omega(Y)(\nabla_XF)Z
=\{(\nabla_X\omega)Y\} FZ+\omega(Y)\omega(X)FZ
\]
for all $X,Y,Z\in T_xM$.
It follows that 
\[
(R(X,Y)F)Z=(\nabla_{X,Y}F)Z-(\nabla_{Y,X}F)Z=0
\]
and so $R\cdot F=0$ at all such $x\in\mathcal U_1$. By a standard topological argument, we conclude that $R\cdot F=0$ on $M$.
\end{proof}

The following result can be obtained immediately from Theorem~\ref{thm:main} and Theorem~\ref{thm:recurrent}.
\begin{corollary}\label{cor:recurrent}
There does not exist any recurrent real hypersurface $M$ in $G_2(\mathbb C_{m+2})$, $m\geq3$. 
\end{corollary}



\begin{thebibliography}{Abc}
\bibitem{berndt}
J. Berndt,
Riemannian geometry of complex two-plane Grassmannian,
\emph{Rend. Semin. Mat. Univ. Politec. Torino}  55(1997), 19--83.

\bibitem{berndt-suh}
J. Berndt, Y.J. Suh, Real hypersurfaces in complex two-plane Grassmannians, 
\emph{Monatsh. Math.}
127(1999) 1--14.


\bibitem{berndt-suh2}
J. Berndt, Y.J. Suh, Real hypersurfaces with isometric Reeb flows on real hypersurfaces in complex two-plane Grassmannians, 
\emph{Monatsh. Math.}
137(2002) 87--98.


\bibitem{lobos}
P.M. Chac\'{o}n, G.A. Lobos, Pseudo-parallel Lagrangian submanifolds in complex space forms, 
\emph{Differ. Geom. Appl.} 27(2009), 137--145.


\bibitem{deprez}
J. Deprez, Semi-parallel hypersurfaces,
\emph{Rend. Semin. Mat. Univ. Politec. Torino}, 44(1986), 303-316.

\bibitem{deprez2}
J. Deprez, 
{Semi-parallel surfaces in euclidean space},
\emph{J. Geom.} {25}(1985),  192--200.

\bibitem{dillen}
F. Dillen, Semi-parallel hypersurfaces of a real space form,
\emph{Isr. J. Math.} 75(1991), 193--202.

\bibitem{kim-lee-yang}
S. Kim, H. Lee, H.Y. Yang,
Real hypersurfaces in complex two-plane Grassmannians with recurrent shape operator,
\emph{Bull. Malays. Math. Sci. Soc.} 34(2011), 295--305.

\bibitem{kobayashi}
S. Kobayashi, K. Nomizu, 
\emph{Foundations of differential geometry, Vol 1}, Interscience Publishers, New York, 1963.

\bibitem{kon}
M. Kon, 
{Semi-parallel CR submanifolds in a complex space form},
\emph{Colloq. Math.} {124}(2011), 237--246.

\bibitem{lee-suh}
H. Lee, Y.J. Suh,
Real hypersurfaces of tyep $B$ in complex two-plane Grassmannians related to the Reeb vector,
\emph{Bull. Korean Math. Soc.} 47(2010), 445--561.

\bibitem{lumiste}
\"U. Lumiste, 
\emph{Semiparallel submanifolds in space forms},
Springer Monographs in Mathematics, Springer, New York, 2009.

\bibitem{machado2}
C. Machado, J.D. P\'{e}rez, Y.J. Suh,
Real hypersurfaces in complex two-plane Grassmannians some of whose Jacobi operators are $\xi$ invariant,
\emph{Int. J. Math.} 23(2012), 1250002 (12 pages), DOI: 10.1142/S0129167X1100746X.

\bibitem{machado1}
C. Machado, J.D. P\'{e}rez, Y.J. Suh,
Real hypersurfaces of Codazzi type in complex two-plane Grassmannians,
\emph{Bull. Malays. Math. Sci. Soc.} (in press).

\bibitem{maeda}
S. Maeda, Real hypersurfaces of complex projective spaces,
\emph{Math. Ann.} 263(1983), 473--478.

\bibitem{naitoh}
H. Naitoh, 
{Parallel submanifolds of complex space forms, I, II},
\emph{Nagoya Math. J.} {90}(1983), 85--117; {91}(1983), 119--149.

\bibitem{ryan}
R. Niebergall, P.J. Ryan, Semi-parallel and semi-symmetric real hypersurfaces in complex space forms, 
\emph{Kyungpook Math. J.} 38(1998), 227--234.

\bibitem{ortega}
M. Ortega, 
{Classifications of real hypersurfaces in complex space forms by means of curvature conditions},
\emph{Bull. Belg. Math. Soc. Simon Stevin} {9}(2002), 351--360.



\bibitem{suh1}
Y.J. Suh,
Real hypersurfaces in complex two-plane Grassmannians with parallel shape operator,
\emph{Bull. Aust. Math. Soc.} 67(2003), 493--502.

\bibitem{suh6}
Y.J. Suh, 
Real hypersurfaces in complex two-plane Grassmannians with Reeb parallel Ricci tensor,
\emph{J. Geom. Phys.} 64(2013), 1--11.

\bibitem{suh5}
Y.J. Suh, 
Real hypersurfaces of type $B$ in complex two-plane Grassmannians, 
\emph{Monatsh. Math.}
147(2006) 337--355.

\bibitem{suh2}
Y.J. Suh, 
Recurrent real hypersurfaces in complex two-plane Grassmannians,
\emph{Acta Math. Hungar.} 112(2006), 89--102.

\bibitem{wong}
Y. Wong, Recurrent tensors on a linearly connected differential manifold,
\emph{Trans. Amer. Math. Soc.} 99(1961), 325--341.
\end{thebibliography}
\end{document}